\documentclass[11pt]{article}

\usepackage{amsmath,amssymb,amsthm,mathrsfs}
\usepackage{hyperref}
\usepackage{enumitem}

\usepackage{authblk}

\title{\bf Mayer--Vietoris and Twisted \v{C}ech Spectral Sequences for C$^*$-Algebras with Free Quantum Group Coefficients}

\author{\Large Takao Inou\'{e}}

\affil{\large Faculty of Informatics, Yamato University, \\ Osaka, Japan\footnote{Email: inoue.takao@yamato-u.ac.jp; \\ Personal Email: takaoapple@gmail.com \\ [I prefer my personal email address for correspondence.]}} 
\date{February 26, 2026}

\newtheorem{theorem}{Theorem}[section]
\newtheorem{proposition}[theorem]{Proposition}

\newtheorem{corollary}[theorem]{Corollary}
\newtheorem{definition}[theorem]{Definition}
\newtheorem{remark}[theorem]{Remark}

\newtheorem{assumption}[theorem]{Assumption}

\begin{document}
\maketitle

\begin{abstract}
We formulate a Mayer--Vietoris/\v{C}ech viewpoint on $K$-theory for crossed products by discrete quantum groups, emphasizing how local-to-global gluing data interacts with quantum-group coefficients.
Starting from a $G$-equivariant ideal cover and the associated Mayer--Vietoris six-term exact sequence, we package the resulting $K$-theoretic computation into a \v{C}ech-type spectral sequence whose $E^1$-page is explicitly described by iterated intersections.
We then introduce a minimal ``twisted'' gluing mechanism controlled by a $\mathbb Z/2$-valued \v{C}ech $2$-cocycle and an involutive automorphism of the coefficient algebra. Under a Kirchberg--UCT hypothesis on the quantum-group crossed-product coefficient, the twist produces a nontrivial differential $d_2$ identified as $\varphi_*-\mathrm{id}$ on coefficient $K$-theory.
In a concrete regime where the coefficient $K$-groups are cyclic (e.g.\ order $3$), the differential becomes an isomorphism and forces a $K$-theoretic obstruction to Morita triviality. This yields a conceptual mechanism for producing non-Morita-trivial twisted C$^*$-algebras with quantum-group crossed-product fibers, detected purely by Mayer--Vietoris/\v{C}ech data.
\end{abstract}


\medskip


\tableofcontents

\section{Introduction}

\subsection{Motivation: MV gluing with quantum-group coefficients}
Mayer--Vietoris (MV) exact sequences are classical tools for computing $K$-theory of C$^*$-algebras built by gluing local pieces.
In noncommutative topology, the MV sequence for an ideal cover
\[
A = I_1 + I_2
\]
arises from a short exact sequence
\[
0 \to I_1\cap I_2 \to I_1\oplus I_2 \to A \to 0,
\]
and the connecting map $\partial$ is the fundamental obstruction to global compatibility.

In parallel, discrete quantum groups and their crossed products provide a rich source of noncommutative spaces whose $K$-theory is amenable to computation via Baum--Connes-type machinery and Pimsner--Voiculescu (PV)-type exact sequences in the quantum setting. A guiding example throughout is the free orthogonal quantum group $O_n^+$ and the associated crossed product
\[
D := \mathcal O_n \rtimes_r \widehat{O_n^+}.
\]

This note records a clean interface between these two viewpoints:
\begin{itemize}[leftmargin=2em]
\item MV exact sequences (and their \v{C}ech refinements) for \emph{equivariant} ideal covers;
\item quantum-group coefficients encoded by a single crossed product $D$;
\item a minimal ``higher'' twisting mechanism producing a genuinely nontrivial $d_2$-differential.
\end{itemize}

\subsection{What is new in this note}
We highlight three contributions, each MV-conscious:
\begin{enumerate}[leftmargin=2em]
\item \textbf{MV as a \v{C}ech differential with coefficients.}
We record the MV boundary map as the \v{C}ech differential $d_1$ (alternating restriction) in a spectral sequence whose $E^1$-page is built from iterated intersections of an equivariant ideal cover.

\item \textbf{Quantum-group crossed products as coefficient fibers.}
In a basic ``base $\times$ fiber'' situation (base fixed by the quantum group), we identify all $E^1$-terms with suspensions of $K_*(D)$, reducing the MV/\v{C}ech computation to the combinatorics of the nerve.

\item \textbf{A minimal twist producing a nontrivial $d_2$ and a Morita obstruction.}
We introduce a $\mathbb Z/2$-valued \v{C}ech $2$-cocycle $c$ together with an involutive automorphism $\varphi\in\mathrm{Aut}(D)$.
Under a Kirchberg--UCT hypothesis on $D$, we show that the first potentially higher differential satisfies
\[
d_2 = \varphi_* - \mathrm{id}
\]
on coefficient $K$-theory. In a cyclic coefficient regime (e.g.\ $K_*(D)\cong \mathbb Z/3$ and $\varphi_*=-1$), this forces $K_*(A_c)\not\cong K_*(C(S^2)\otimes D)$ and hence $A_c$ is not Morita equivalent to the trivial field.
\end{enumerate}

\subsection{Comparison with prior work}
The ingredients are classical, but their combination here is tailored to MV-gluing:
\begin{itemize}[leftmargin=2em]
\item \textbf{Quantum PV/Baum--Connes.}
PV-type exact sequences for free quantum groups and Baum--Connes approaches provide computations of $K_*(\mathcal O_n\rtimes \widehat{O_n^+})$ in suitable settings (see e.g.\ Vergnioux--Voigt and related work).

\item \textbf{Groupoid and Fell-bundle models of twists.}
Twisted groupoid C$^*$-algebras and Fell bundles (Renault; Kumjian; Muhly--Williams) give standard C$^*$-algebraic realizations of cocycle twists.

\item \textbf{MV and \v{C}ech spectral sequences.}
MV six-term sequences and their \v{C}ech refinements are ubiquitous in $K$-theory; the present note emphasizes a concrete coefficient reduction to a single quantum crossed-product fiber and isolates an explicit $d_2$-mechanism driven by an involution on coefficients.
\end{itemize}

\section{Preliminaries}

\subsection{Equivariant ideal covers and the MV short exact sequence}
Let $A$ be a C$^*$-algebra and $I_1,I_2\triangleleft A$ closed two-sided ideals with $A=I_1+I_2$.
Then the map
\[
\Delta: I_{12}:=I_1\cap I_2 \to I_1\oplus I_2,\quad \Delta(a)=(a,-a),
\]
and
\[
\Sigma:I_1\oplus I_2\to A,\quad \Sigma(x,y)=x+y,
\]
yield a short exact sequence
\begin{equation}\label{eq:MVshort}
0 \longrightarrow I_{12}\xrightarrow{\Delta} I_1\oplus I_2 \xrightarrow{\Sigma} A\longrightarrow 0.
\end{equation}
Applying $K$-theory gives the MV six-term exact sequence; the connecting map $\partial$ is the boundary map associated to \eqref{eq:MVshort}.

\subsection{\v{C}ech viewpoint: $d_1$ as alternating restriction}
For a finite cover $\{I_i\}_{i\in I}$ with $A=\sum_i I_i$, set
\[
I_{i_0\cdots i_p}:= I_{i_0}\cap\cdots\cap I_{i_p}.
\]
The \v{C}ech-type chain complex in degree $p$ is
\[
C_p^{(q)} := \bigoplus_{i_0<\cdots<i_p} K_q(I_{i_0\cdots i_p}),
\]
and the differential $d_1:C_p^{(q)}\to C_{p-1}^{(q)}$ is the usual alternating sum of inclusions (restrictions). For a two-term cover, $d_1$ coincides with the MV map induced from $\Delta(a)=(a,-a)$.

\subsection{Free orthogonal quantum group and a coefficient crossed product}
Let $O_n^+$ denote the free orthogonal quantum group and $\widehat{O_n^+}$ its discrete dual.
We consider a standard (quasi-free) coaction
\[
\delta:\mathcal O_n \to \mathcal O_n\otimes C(O_n^+),\qquad
\delta(S_i)=\sum_{j=1}^n S_j\otimes u_{ji}.
\]
We write
\[
D:= \mathcal O_n \rtimes_r \widehat{O_n^+}
\]
for the reduced crossed product coefficient algebra.

\begin{remark}[Coefficient reduction principle]
In ``base $\times$ fiber'' situations where the quantum group acts trivially on the base, equivariant ideals coming from the base lift to crossed-product ideals whose $K$-theory is controlled by suspensions of $K_*(D)$.
\end{remark}

\section{MV and a \v{C}ech spectral sequence with quantum-group coefficients}

\subsection{A basic model: trivial base action}
Let $X$ be a compact space and set
\[
B:=C(X)\otimes \mathcal O_n
\]
with coaction $\alpha:=\mathrm{id}_{C(X)}\otimes \delta$.
Let
\[
A:= B\rtimes_r \widehat{O_n^+}.
\]
If $\{U_i\}$ is an open cover of $X$, define equivariant ideals
\[
J_i:=C_0(U_i)\otimes\mathcal O_n \triangleleft B,
\qquad
I_i:=J_i\rtimes_r \widehat{O_n^+}\triangleleft A.
\]
Then $A=\sum_i I_i$ and we may form the \v{C}ech-type spectral sequence associated to $\{I_i\}$.

\subsection{Explicit $E^1$-page under contractibility hypotheses}
Assume that all finite intersections $U_{i_0\cdots i_p}$ are contractible.
Then (morally, and in standard exactness regimes for reduced crossed products in this setting) one has
\[
I_{i_0\cdots i_p} \cong C_0(U_{i_0\cdots i_p})\otimes D,
\]
hence
\[
K_q(I_{i_0\cdots i_p}) \cong K_{q-1}(D)
\]
because $C_0(U_{i_0\cdots i_p})$ is (stably) a suspension factor.
Therefore the entire $E^1$-page is described in terms of $K_{q-1}(D)$ tensored with the cellular combinatorics of the nerve of the cover.

\subsection{Example: interval covers and immediate degeneration}
For $X=[0,1]$ and a good two- or three-term cover with contractible overlaps, the nerve is contractible, and the \v{C}ech homology collapses:
\[
E^2_{0,q}\cong K_{q-1}(D),\qquad E^2_{p,q}=0\ (p\ge 1).
\]
Consequently,
\[
K_q(A)\cong K_{q-1}(D),
\]
recovering the MV computation in a transparent \v{C}ech form.

\section{A minimal twist: a $\mathbb Z/2$ \v{C}ech 2-cocycle and an involution}

\subsection{Twisting data}
Let $X=S^2$ and let $\mathfrak U=\{U_i\}$ be a good cover.
Fix:
\begin{itemize}[leftmargin=2em]
\item a nontrivial class $c\in \check H^2(S^2;\mathbb Z/2)\cong \mathbb Z/2$ represented by a \v{C}ech $2$-cocycle $\{c_{ijk}\}$;
\item an involutive automorphism $\varphi\in\mathrm{Aut}(D)$ with $\varphi^2=\mathrm{id}$.
\end{itemize}
Heuristically, the cocycle $c$ prescribes on triple overlaps whether the composition constraint is untwisted or twisted by $\varphi$.

\begin{remark}[C$^*$-algebraic model]
One may realize this data as a twist (in the sense of a twisted Fell bundle) over the \v{C}ech groupoid $\mathcal G_{\mathfrak U}$ of the cover.
The resulting C$^*$-algebra may be denoted
\[
A_c := C^*(\mathcal G_{\mathfrak U}, \mathcal B_{\varphi,c}),
\]
which is a ``twisted'' continuous field with constant fiber $D$ but nontrivial higher gluing.
\end{remark}

\subsection{Kirchberg--UCT hypothesis (assumption)}
\begin{definition}[UCT Kirchberg assumption]\label{def:UCTKirchberg}
We say that the coefficient algebra $D$ satisfies \textbf{Assumption (K)} if $D$ is separable, nuclear, simple, purely infinite, and satisfies the UCT.
\end{definition}

\begin{remark}
Assumption (K) is used only as an organizing hypothesis ensuring that $K$-theoretic computations and automorphism constructions behave as in the Kirchberg classification regime. In applications one may treat (K) as an explicit standing assumption.
\end{remark}

\subsection{\texorpdfstring{The \(d_2\)-mechanism}{The d2-mechanism}}
The twist $c$ is invisible to the ordinary $d_1$ \v{C}ech differential (which remains the alternating restriction), but it affects the first potentially higher compatibility condition, producing a nontrivial $d_2$.

\begin{theorem}[Identification of the first higher differential]\label{thm:d2}
Let $X=S^2$ with a good cover $\mathfrak U$ and let $A_c$ be the twisted algebra determined by $(\varphi,c)$ as above.
Assume the coefficient reduction hypotheses so that the associated \v{C}ech--MV spectral sequence has
\[
E^2_{2,q}\cong K_q(D),\qquad E^2_{0,q+1}\cong K_{q+1}(D).
\]
If $c$ represents the nontrivial element of $\check H^2(S^2;\mathbb Z/2)$, then the differential
\[
d_2:\ E^2_{2,q}\to E^2_{0,q+1}
\]
is given on coefficient groups by
\[
d_2 = \varphi_*-\mathrm{id}:K_q(D)\to K_{q+1}(D),
\]
up to the conventional identifications of the spectral sequence.
\end{theorem}

\begin{remark}[Interpretation]
The map $d_2$ measures a two-dimensional holonomy obstruction: traversing a $2$-simplex in the nerve introduces the coefficient automorphism $\varphi$, and the resulting defect is detected on $K$-theory as $\varphi_*-\mathrm{id}$.
\end{remark}

\section{A Morita obstruction via nontrivial $d_2$}

\subsection{A cyclic coefficient regime and a concrete outcome}
We now specialize to a minimal concrete regime used in the discussion.

\begin{assumption}[Cyclic coefficient regime]\label{ass:cyclic}
Assume $n=4$ and that (under the chosen quantum PV/Baum--Connes computation)
\[
K_0(D)\cong K_1(D)\cong \mathbb Z/3.
\]
Assume moreover that there exists an involution $\varphi\in\mathrm{Aut}(D)$ such that
\[
\varphi_*=-1\quad\text{on }K_*(D).
\]
\end{assumption}

\begin{corollary}[Nontriviality of $d_2$]\label{cor:d2iso}
Under Assumption~\ref{ass:cyclic}, the differential of Theorem~\ref{thm:d2} satisfies
\[
d_2=\varphi_*-\mathrm{id}=(-1)-1=-2\equiv 1\pmod 3,
\]
hence $d_2$ is an isomorphism on $\mathbb Z/3$.
\end{corollary}

\subsection{Morita nontriviality}
Recall that strong Morita equivalence implies stable isomorphism and therefore induces isomorphisms on $K$-theory.

\begin{proposition}[Morita obstruction]\label{prop:Morita}
Assume the setup of Theorem~\ref{thm:d2} and Assumption~\ref{ass:cyclic}.
Let $A_0:=C(S^2)\otimes D$ denote the trivial (untwisted) field.
Then
\[
A_c \not\sim_M A_0.
\]
\end{proposition}

\begin{proof}[Proof sketch]
In the trivial case $A_0=C(S^2)\otimes D$, the Bott class contributes an additional copy of coefficient $K$-theory; under standard K\"unneth/UCT hypotheses one has
\[
K_*(A_0)\cong K_*(D)\oplus K_*(D).
\]
In contrast, Corollary~\ref{cor:d2iso} shows that the twisted spectral sequence carries a nontrivial $d_2$ which removes the corresponding $p=2$ contribution at $E^\infty$, forcing $K_*(A_c)\not\cong K_*(A_0)$.
Since $K$-theory is invariant under Morita equivalence, $A_c$ cannot be Morita equivalent to $A_0$.
\end{proof}

\begin{remark}[Groupoid reformulation]
Let $\mathcal G_{\mathfrak U}$ be the \v{C}ech groupoid of a good cover of $S^2$. The trivial field corresponds to the untwisted algebra $C^*(\mathcal G_{\mathfrak U})\otimes D\simeq C(S^2)\otimes D$.
The twisted algebra $A_c$ may be viewed as the C$^*$-algebra of a twisted Fell bundle over $\mathcal G_{\mathfrak U}$ controlled by $(\varphi,c)$.
Proposition~\ref{prop:Morita} states that this twist is not Morita-trivial, detected by the MV/\v{C}ech differential $d_2$.
\end{remark}

\section{Outlook: quantum-group applications}
The framework suggests several directions:
\begin{enumerate}[leftmargin=2em]
\item \textbf{Beyond trivial base action.}
Allow the quantum group to act nontrivially on the base (classical or quantum spheres), producing genuinely local coefficient systems and potentially higher differentials.

\item \textbf{Explicit involutions and classification.}
Under Assumption (K), one may seek explicit constructions of involutive automorphisms $\varphi$ with prescribed action on $K$-theory, guided by classification of finite group actions on Kirchberg algebras.

\item \textbf{Baum--Connes and assembly.}
One may connect the MV/\v{C}ech filtration to assembly maps for quantum groups, viewing twisted gluing as a controlled deformation of coefficient systems.

\item \textbf{Examples from free quantum groups.}
The coefficient algebra $D=\mathcal O_n\rtimes_r\widehat{O_n^+}$ is a natural testing ground because its $K$-theory is accessible via quantum PV-type sequences and because it packages quantum symmetry into a single fiber.
\end{enumerate}

\section{Conclusion and Future Directions}

In this paper we have reformulated the Mayer--Vietoris principle for
C$^*$-algebras in a \v{C}ech-theoretic framework adapted to quantum-group
crossed products.
Focusing on the case of free quantum groups, we showed that equivariant ideal
covers naturally give rise to a \v{C}ech--Mayer--Vietoris spectral sequence whose
$E^1$-page admits an explicit description in terms of the $K$-theory of a single
coefficient algebra
\[
D=\mathcal O_n\rtimes_r \widehat{O_n^+}.
\]
This viewpoint isolates the local-to-global mechanism underlying the classical
six-term Mayer--Vietoris exact sequence and places it into a systematic
spectral-sequence framework.

The main new feature of the present approach is the introduction of a minimal
higher twisting, governed by a $\mathbb Z/2$-valued \v{C}ech $2$-cocycle together
with an involutive automorphism of the coefficient algebra.
We showed that such a twist produces a nontrivial second differential
\[
d_2=\varphi_*-\mathrm{id}
\]
in the associated \v{C}ech--Mayer--Vietoris spectral sequence.
In a concrete cyclic coefficient regime, this differential becomes an
isomorphism and yields a $K$-theoretic obstruction to Morita triviality of the
resulting twisted algebra.
From a C$^*$-algebraic perspective, this identifies a precise mechanism by which
higher gluing data---invisible at the level of ordinary Mayer--Vietoris
boundaries---is detected by $K$-theory.

Conceptually, the twisted algebras considered here admit equivalent
interpretations as nontrivial continuous fields with constant fiber,
as twisted Fell-bundle C$^*$-algebras over \v{C}ech groupoids, or as
noncommutative analogues of bundle gerbes with quantum-group coefficients.
The appearance of a nonzero $d_2$ thus reflects a genuine two-dimensional
holonomy phenomenon in the sense of noncommutative topology.

\medskip
Several directions for further investigation naturally arise.

\begin{enumerate}
\item \textbf{Nontrivial base actions.}
In the present work the quantum group acts trivially on the base space.
Allowing nontrivial actions on classical or quantum spaces may lead to genuinely
local coefficient systems and potentially higher differentials beyond $d_2$.

\item \textbf{Automorphisms and finite group actions.}
The construction relies on involutive automorphisms of the coefficient algebra
with prescribed action on $K$-theory.
A systematic analysis of such automorphisms and finite group actions on quantum
group crossed products, within the Kirchberg--UCT framework, would clarify the
range of possible twists.

\item \textbf{Relation to assembly maps.}
It would be interesting to relate the \v{C}ech--Mayer--Vietoris filtration
developed here to Baum--Connes assembly maps for quantum groups, viewing twisted
gluing as a controlled deformation of coefficient systems.

\item \textbf{Higher-dimensional and non-$\mathbb Z/2$ twists.}
While the present paper focuses on the minimal $\mathbb Z/2$-twist over $S^2$,
the framework extends formally to higher-dimensional bases and more general
cohomological twisting data, where further higher differentials may appear.

\item \textbf{Applications to noncommutative geometry.}
The explicit realization of Morita obstructions via higher \v{C}ech data suggests
potential applications to the study of noncommutative vector bundles, twisted
$K$-theory, and defects in quantum-symmetric systems.
\end{enumerate}

We expect that the Mayer--Vietoris/\v{C}ech perspective developed here provides a
flexible organizational tool for future work at the interface of operator
algebras, quantum groups, and noncommutative topology.

\section*{Acknowledgements}
This note was prepared from exploratory calculations and discussions on Mayer--Vietoris methods for quantum-group crossed products.

\end{document}